\RequirePackage{fix-cm}

\documentclass[smallcondensed]{svjour3}

\smartqed
\usepackage{float}
\usepackage{color}
\usepackage{cite}

\usepackage{comment}
\usepackage{amsmath}
\usepackage{amssymb}
\usepackage{multirow}

\usepackage{epsfig}
\usepackage{float}
\usepackage{verbatim}
\usepackage{graphicx}
\usepackage{subfigure}
\usepackage{mathrsfs}

\usepackage{extarrows}
\usepackage{color}
\usepackage[ruled,vlined]{algorithm2e}
\usepackage{hyperref}
\usepackage{algorithmic}

\usepackage{graphicx}
\usepackage[ruled,vlined]{algorithm2e}

\usepackage[misc]{ifsym}

\newcommand{\bq }{\begin{equation}}
\newcommand{\eq }{\end{equation}}

\numberwithin{equation}{section}

\usepackage{tikz,xcolor,hyperref}

\definecolor{lime}{HTML}{A6CE39}
\DeclareRobustCommand{\orcidicon}{
	\begin{tikzpicture}
	\draw[lime, fill=lime] (0,0)
	circle [radius=0.16]
	node[white] {{\fontfamily{qag}\selectfont \tiny ID}};
	\draw[white, fill=white] (-0.0625,0.095)
	circle [radius=0.007];
	\end{tikzpicture}
	\hspace{-2mm}
}
\foreach \x in {A, ..., Z}{\expandafter\xdef\csname orcid\x\endcsname{\noexpand\href{https://orcid.org/\csname orcidauthor\x\endcsname}
			{\noexpand\orcidicon}}
}

\begin{document}

\title{Optimal error estimate of accurate second-order scheme for Volterra integrodifferential equations with tempered multi-term kernels}

\author{Wenlin Qiu}

\institute{
W. Qiu (\Letter) \at
MOE-LCSM, School of Mathematics and Statistics, Hunan Normal University, Changsha, Hunan 410081, P. R. China\\
\email{qwllkx12379@163.com}
}
\date{Received: date / Accepted: date}

\maketitle

\begin{abstract}
In this paper, we investigate and analyze numerical solutions for the Volterra integrodifferential equations with tempered multi-term kernels. Firstly we derive some regularity estimates of the exact solution. Then a temporal-discrete scheme is established by employing Crank-Nicolson technique and product integration (PI) rule for discretizations of the time derivative and tempered-type fractional integral terms, respectively, from which, nonuniform meshes are applied to overcome the singular behavior of the exact solution at $t=0$. Based on deduced regularity conditions, we prove that the proposed scheme is unconditionally stable, and possesses accurately temporal second-order convergence in $L_2$-norm. Numerical examples confirm the effectiveness of the proposed method.
\end{abstract}

\section{Introduction}
In this work, we consider the following nonlocal evolution equation with tempered multi-term kernels
   \begin{equation}\label{eq1.1}
   \begin{split}
      \frac{\partial u}{\partial t} &+ Au + \sum\limits_{j=1}^m (\beta_j *B_ju)(t)   = f(t),   \quad   t >0,  \quad 1\leq m <\infty, \\
      &  u(0)=u_0,
   \end{split}
   \end{equation}
in which $A$ is a self-adjoint positive-definite linear operator, not necessarily bounded operator, with compact inverse, defined on a dense subspace $D(A)$ of the Hilbert space $\mathbf{H}$, and $B_j$ is the self-adjoint linear operator in $\mathbf{H}$ such that $D(B_j)\supset D(A)$, $1 \leq j \leq m$. Letting
   \begin{equation*}
   \begin{split}
        \|w\|_q=\|A^{q/2}w\|=(A^qw,w)^{1/2},
   \end{split}
   \end{equation*}
   where $(\cdot,\cdot)$ is the inner product in $\mathbf{H}$ with the norm $\|\cdot\|$, we suppose that in this paper the operator $A$ dominates $B_j$ in the sense that \cite{Larsson}
   \begin{equation}\label{eq1.2}
   \begin{split}
        \left| (B_jw,v) \right| \leq C \|w\|_q \|v\|_{\vartheta}, \quad \vartheta=0,1,2, \quad \vartheta+q=2.
   \end{split}
   \end{equation}
Define the convolution
   \begin{equation*}
   \begin{split}
       (\beta_j *\varphi)(t):=\int_0^t \beta_j(t-s) \varphi(s) ds,  \quad 1 \leq j \leq m, \quad   t >0,
   \end{split}
   \end{equation*}
see \cite{Podlubny,Xu1}, where the tempered singular kernels
   \begin{equation}\label{eq1.3}
   \begin{split}
         \beta_j(t)=\frac{e^{-\kappa t}t^{\alpha_j -1}}{\Gamma(\alpha_j)},  \quad  \kappa\geq 0, \quad 0<\alpha_j <1,
   \end{split}
   \end{equation}
and $\Gamma(\cdot)$ denotes the Euler's Gamma function; see \cite{Podlubny}. Problems of type \eqref{eq1.1} with the single-term kernel ($m=1, \kappa=0$) can be considered as the model arising in heat transfer theory with memory, population dynamics and viscoelastic materials, see \cite{Friedman,Heard} and references therein; with $m=1$ and $\kappa>0$, \eqref{eq1.1} is called the tempered fractional integro-differential equation, in which the tempered fractional integral of Brownian motion, called tempered fractional Brownian motion, can show semi-long-range correlation, and then the increment of this process is called tempered fractional Gaussian noise, which affords a useful new stochastic model for the wind speed data, see \cite{Sabzikar}. Furthermore, when $A=0$, problem \eqref{eq1.1} with self-adjoint boundary conditions appears in a linear model for heat flow in a rectangular, orthotropic material with the memory \cite{Carslaw,Hannsgen,MacCamy}, from which the axes of orthotropy are parallel to the edges of the rectangle.

Due to its broad and practical applications, many scholars have carried out a series of theoretical and numerical studies regarding the problem of type \eqref{eq1.1}, e.g., in terms of theoretical research with $A=0$ and $f=0$, Hannsgen and Wheeler \cite{Hannsgen} first studied the asymptotic behavior as $t\rightarrow \infty$ of solutions of \eqref{eq1.1}, combined with a resolvent formula, which gave the long-time estimate of solutions with the weight function. Then based on \cite{Hannsgen}, Noren \cite{Noren1} proved the long-time estimate of solutions without the weight function by giving variable hypothesis of kernel functions $\beta_j(t)$. Subsequently, Noren \cite{Noren1} derived that $u'(t)$ is integrable uniformly with respect to the parameters of solutions under some assumptions regarding kernel functions $\beta_j(t)$. With the maturity of theoretical studies, numerical researches have gradually developed. Xu proved the weighted $L^1$ asymptotic stability of the numerical solutions by temporal backward Euler (BE) method and first-order convolution quadrature rule. Hereafter, Xu considered the weighted $L^1$ asymptotic convergence of the numerical solutions by same temporal approximation. After that, in order to improve accuracy to the second order for time, Xu proved the weighted $L^1$ asymptotic stability \cite{Xu4} and weighted $L^1$ asymptotic convergence \cite{Xu5} by second-order backward differentiation formula (BDF) and second-order convolution quadrature, from which, temporal convergence can not reach accurate second order due to the singular behavior of the exact solution at the initial point $t=0$. Besides, some studies of \eqref{eq1.1} with $m=2$ were considered recently, see \cite{Cao,Qiu}, which only achieved temporal first-order accuracy; also, there has been a lot of excellent work on similar multi-term problems, see \cite{Jin,Liu,Zhou}. Thus, the limitations of existing methods motivate us to carry out the following research.

\vskip 0.2mm
The major aim of this work is to consider an accurate second-order numerical scheme for Volterra integrodifferential equations with tempered-type multi-term singular kernels. This time-discrete scheme is constructed by the Crank-Nicolson method and PI rule on the nonuniform meshes. In addition, the main contribution of this work can be summarized in the following aspects:
 $(i)$ Based on certain suitable assumptions, we derive the regularity estimates of the exact solution of problem \eqref{eq1.1};
$(ii)$ We employ the graded meshes in time to establish the discrete scheme, which can compensate for the singular behaviour of the exact solutions at $t=0$;
$(iii)$ The energy method is utilized to deduce the stability and convergence of the proposed scheme;
$(iv)$ Provided numerical examples satisfying the regularity substantiate the theoretical analyses, which can reflect and illustrate the accurate second-order accuracy for time by choosing the grading index $\gamma\geq 2/(1+\alpha)$, $\alpha=\min\limits_{1 \leq j \leq m}\{\alpha_j\}$.

\vskip 0.2mm
The rest of this article is arranged as follows. In section \ref{sec2}, the regularity of exact solutions is deduced on the $L_2$ space. Section \ref{sec3} gives some notations and formulates a time-discrete scheme. Then the stability and convergence of proposed scheme are proved in section \ref{sec4}. Section \ref{sec5} provides several numerical examples to verify our analysis. Finally, some concluding remarks are presented in section \ref{sec6}.

\vskip 0.2mm
 Throughout this paper, $C$ denotes a generic positive constant that is independent of the space-time step sizes, and may be not necessarily the same on each occurrence.

\section{Regularity of exact solutions} \label{sec2} For the further analysis, denoting
   \begin{equation}\label{eq2.1}
   \begin{split}
       v(t)=e^{\kappa t}u(t), \quad g(t)=e^{\kappa t}f(t),  \quad  t > 0,
   \end{split}
   \end{equation}
then \eqref{eq1.1} is converted to
   \begin{equation}\label{eq2.2}
   \begin{split}
      \frac{\partial v}{\partial t} &+ Av + \sum\limits_{j=1}^m(\omega_{\alpha_j}*B_jv)(t) - \kappa v  = g(t),   \quad   t >0,  \quad m \geq 1, \\
      &  v(0)=u_0,
   \end{split}
   \end{equation}
where the Abel kernel $\omega_{\alpha_j}(t)=\frac{t^{\alpha_j-1}}{\Gamma(\alpha_j)}$ with $1 \leq j \leq m$.

\vskip 0.1in
Below, we shall give some regularity estimates of the exact solution of \eqref{eq1.1}. That is, we wish to present that
   \begin{equation}\label{eq2.3}
   \begin{split}
        & t\|Av''(t)\| + \|v''(t)\| +  \|Av'(t)\| \leq C t^{\alpha-1}, \quad \alpha=\min\limits_{1 \leq j \leq m}\{\alpha_j\}, \\
        & t\|Au''(t)\| + \|u''(t)\| +  \|Au'(t)\| \leq C e^{-\kappa t}t^{\alpha-1},  \quad   t \rightarrow 0^{+}.
   \end{split}
   \end{equation}
Later $\|\cdot\|_{L_{p}}$ denotes the norm on the Sobolev space $L_{p}(\Omega)$ with $1< p \leq \infty$; besides, Lipschitz domain $\Omega\subseteq \mathbb{R}^d$, $d\geq 1$. In fact, before establishing the regularity estimates of $u$, we only need to consider $v$ defined in \eqref{eq2.2}. By splitting the solution of \eqref{eq2.2} be into $v(t)=v_1(t)+v_2(t)$, which implies that for $t > 0$, we shall further solve
   \begin{equation}\label{eq2.4}
   \begin{split}
       v_1'(t) + Av_1(t) = g(t),   \quad v_1(0) = u_0,
   \end{split}
   \end{equation}
   and
   \begin{equation}\label{eq2.5}
   \begin{split}
       v_2'(t) + Av_2(t) = - \sum\limits_{j=1}^{m} \int_0^t \omega_{\alpha_j}(t-s) B_j v(s) ds + \kappa v(t),   \quad v_2(0) = 0,
   \end{split}
   \end{equation}
respectively. Then, we rewrite the solution of \eqref{eq2.4} as
   \begin{equation}\label{eq2.6}
   \begin{split}
       v_1(t) = E(t)u_0 + \int_0^t E(t-s) g(s) ds,
   \end{split}
   \end{equation}
where $E(t)=\exp(-At)$ is the analytic semigroup in space $L_{p}(\Omega)$, generated by $-A$ (cf.~\cite{Pazy}), such that
   \begin{equation}\label{eq2.7}
   \begin{split}
        \|E(t)w\|_{L_{p}} + t\|AE(t)w\|_{L_{p}} + t^2\|A^2E(t)w\|_{L_{p}} \leq C\|w\|_{L_{p}}, \quad t\in (0,T]
   \end{split}
   \end{equation}
for any $w\in L_{p}(\Omega)$. Then similar to \cite{Larsson,Mustapha}, we introduce the following key lemmas.

\begin{lemma}\label{lemma2.1}
Assume that $p\in (1,\infty)$ and $\varphi(t)\in L_{p}(\Omega)$, $t>0$. Suppose that $\|\varphi(t)\|_{L_{p}}\leq Ct^{\alpha-1}$ and $\|\varphi'(t)\|_{L_{p}}\leq Ct^{\alpha-2}$ with $\alpha=\min\limits_{1 \leq j \leq m}\{\alpha_j\}$. Then it holds that
   \begin{equation*}
   \begin{split}
        \left\| \int_0^t AE(t-s) \varphi(s) ds  \right\|_{L_{p}} \leq Ct^{\alpha-1}, \quad  \alpha\in (0,1).
   \end{split}
   \end{equation*}
\end{lemma}
\begin{proof} The desired result can be proved similarly by \cite[Lemma 5.1]{Larsson}.
\end{proof}

\begin{lemma}\label{lemma2.2}
Let $p\in (1,\infty)$ and $\varphi(t)\in L_{p}(\Omega)$ for $t>0$, then
   \begin{equation*}
   \begin{split}
        \left\| \int_0^t AE(t-s) \int_0^s \omega_{\alpha_j}(s-\sigma)\varphi(\sigma)d\sigma  ds  \right\|_{L_{p}} \leq C \int_0^t \omega_{\alpha_j}(t-\sigma) \|\varphi(\sigma)\|_{L_{p}} d\sigma.
   \end{split}
   \end{equation*}
\end{lemma}
\begin{proof}
  Exchanging the order of integration and swapping variables, we have
   \begin{equation*}
   \begin{split}
       \int_0^t AE(t-s) \int_0^s \omega_{\alpha_j}(s-\sigma)\varphi(\sigma)d\sigma  ds  = \int_0^t \left[ \int_0^{t-\sigma} AE(t-\sigma-\vartheta) \omega_{\alpha_j}(\vartheta) d \vartheta \right]    \varphi(\sigma) d\sigma.
   \end{split}
   \end{equation*}
Then an application of Lemma \ref{lemma2.1} completes the proof.
\end{proof}

\begin{lemma}\label{lemma2.3}
Suppose that $\|Au_0-g(0)\| + \|g(t)\| \leq C$ and $\|(Au_0-g(0))E'(t)\| + \|g'(t)\| + t\|g''(t)\|  \leq Ct^{\alpha-1}$ with $\alpha=\min\limits_{1 \leq j \leq m}\{\alpha_j\}$. Then $v_1(t)$ denoted by \eqref{eq2.4} satisfies
   \begin{equation*}
   \begin{split}
         \|Av_1(t)\|  + \|v_1'(t)\|\leq C,   \quad   \|Av_1'(t)\|   + \|v_1''(t)\|  \leq C t^{\alpha-1}.
   \end{split}
   \end{equation*}
\end{lemma}
\begin{proof}
    We can finish the proof analogously by \cite[Lemma 2.3]{Mustapha}.
\end{proof}

\vskip 0.1in
Below we derive the regularity estimate with respect to $v(t)$.
\begin{theorem}\label{theorem2.4}
Based on the assumptions of Lemma \ref{lemma2.3}, then the solution $v(t)$ of \eqref{eq2.2} satisfies that
   \begin{equation*}
   \begin{split}
       \|v'(t)\|   + \|Av(t)\| \leq C\|u_0\|, \quad 0< t \leq T.
   \end{split}
   \end{equation*}
   \end{theorem}
\begin{proof}
  By \eqref{eq2.5}-\eqref{eq2.7} and the assumption \eqref{eq1.2} ($A$ dominates $B_j$), then $v_2(t)$ given by \eqref{eq2.5} satisfies
   \begin{equation*}
   \begin{split}
         \|v_2'(t)\| &\leq  \sum\limits_{j=1}^{m} \left\|\int_0^t \omega_{\alpha_j}(t-s) (B_jA^{-1})A v(s) ds\right\|  \\
         & +  \|Av_2(t)\| + \kappa \Big( \|v_1(t)\| + \|v_2(t)\| \Big) \\
         & \leq  \sum\limits_{j=1}^{m} \int_0^t \omega_{\alpha_j}(t-s) \left\|A v(s)\right\| ds +  \|Av_2(t)\| \\
         & + \kappa \left( \|u_0\| + \int_{0}^{t}\|g(s)\|ds \right) + \kappa\int_{0}^{t}\|v_2'(s)\|ds.
   \end{split}
   \end{equation*}
Then using Duhamel's principle and \eqref{eq2.3}, we deduce
   \begin{equation}\label{eq2.8}
   \begin{split}
        v_2(t) & = - \int_0^t E(t-s) \sum\limits_{j=1}^{m} \int_0^s \omega_{\alpha_j}(s-\sigma)B_jv(\sigma)d\sigma ds  \\
               & - \kappa \int_0^t E(t-s)v(s)ds := v_{21} + v_{22}.
   \end{split}
   \end{equation}
From Lemma \ref{lemma2.2} and Lemma \ref{lemma2.3}, we yield
   \begin{equation*}
   \begin{split}
        \|Av_{21}\| & \leq \left\| \int_0^t A E(t-s) \sum\limits_{j=1}^{m} \int_0^s \omega_{\alpha_j}(s-\sigma)(B_jA^{-1}) Av(\sigma)d\sigma ds  \right\| \\
         & \leq C \sum\limits_{j=1}^{m}\int_0^t \omega_{\alpha_j}(t-\sigma) \left(\left\|Av_1(\sigma)\right\| + \left\|Av_2(\sigma)\right\|  \right)  d\sigma  \\
         & \leq C \sum\limits_{j=1}^{m} \left(  t^{\alpha_j} + \int_0^t \omega_{\alpha_j}(t-\sigma) \left\|Av_2(\sigma)\right\|d\sigma   \right),
   \end{split}
   \end{equation*}
and \eqref{eq2.7} gives
   \begin{equation*}
   \begin{split}
        \|Av_{22}\| & \leq  \kappa \left\| \int_0^t  E(t-s) Av(s) ds  \right\| \leq C \int_0^t \| Av(s)\| ds.
   \end{split}
   \end{equation*}
Therefore, we have
   \begin{equation*}
   \begin{split}
        \|Av_{2}(t)\| & \leq C\left(  t^{\alpha} + \int_0^t \| Av(s)\| ds \right) + \int_0^t \omega_{\alpha}(t-\sigma) \left\|Av_2(\sigma)\right\|d\sigma,
   \end{split}
   \end{equation*}
from which, an application of Gr\"{o}nwall's lemma (see \cite[Lemma 1]{Chen}) yields
   \begin{equation*}
   \begin{split}
        \|Av_{2}(t)\| & \leq C\left(  t^{\alpha} + \int_0^t \| Av(s)\| ds \right),
   \end{split}
   \end{equation*}
which follows from Lemma \ref{lemma2.3} that $\|Av(t)\| \leq C + C\int_0^t \| Av(s)\| ds$, then we have
$ \|Av(t)\|  \leq C$ and  $\|Av_{2}(t)\| \leq  C t^{\alpha}$. Thus, we can get
   \begin{equation*}
   \begin{split}
         \|v_2'(t)\| & \leq  C \left(\sum\limits_{j=1}^{m} t^{\alpha_j} +  \|u_0\|\right) + \kappa\int_{0}^{t}\|v_2'(s)\|ds,
   \end{split}
   \end{equation*}
then Gr\"{o}nwall's lemma and Lemma \ref{lemma2.3} gives
   \begin{equation}\label{eq2.9}
   \begin{split}
        \|v'(t)\|  &\leq \|v_1'(t)\| + \|v_2'(t)\| \leq C\|u_0\|,
   \end{split}
   \end{equation}
which completes the proof.
\end{proof}

\vskip 0.1in
We will give other estimates of $v(t)$ regarding the second time derivative.
\begin{theorem}\label{theorem2.5}
Based on the assumptions of Theorem \ref{theorem2.4}. Then it follows that
   \begin{equation*}
   \begin{split}
        \left\| v''(t) \right\| + \left\| Av'(t) \right\|   \leq C t^{\alpha -1}\left\| u_0\right\|_2, \quad \alpha=\min\limits_{1 \leq j \leq m}\{\alpha_j\}, \quad t\in (0,T].
   \end{split}
   \end{equation*}
   \end{theorem}
\begin{proof}
  Differentiate \eqref{eq2.5} to get
   \begin{equation}\label{eq2.10}
   \begin{split}
       v_2''(t) + Av_2'(t) =  \kappa v'(t) - \sum\limits_{j=1}^{m}\omega_{\alpha_j}(t)B_ju_0  - \sum\limits_{j=1}^{m}  \int_0^t \omega_{\alpha_j}(s)B_j   \frac{\partial v(t-s)}{\partial t} ds,
   \end{split}
   \end{equation}
where using \eqref{eq2.4}-\eqref{eq2.5}, we have $v_2'(0)=v_1(0)=u_0$, thus Duhamel's principle gives
   \begin{equation}\label{eq2.11}
   \begin{split}
       v_2'(t) & = - \left( E(t)u_0 - \kappa \int_0^t E(t-s) \big( v_1'(s) + v_2'(s) \big)   \right)   \\
       & - \sum\limits_{j=1}^{m}\int_0^t E(t-s)\omega_{\alpha_j}(s) (B_jA^{-1}) Au_0  ds \\
       & - \sum\limits_{j=1}^{m}\int_0^t E(t-s) \int_0^s \omega_{\alpha_j}(s-\sigma) B_j v_2'(\sigma) d\sigma ds  \\
       & - \sum\limits_{j=1}^{m}\int_0^t E(t-s) \int_0^s \omega_{\alpha_j}(s-\sigma) B_j v_1'(\sigma) d\sigma ds  \\
       & := R_{21} + R_{22} + R_{23} + R_{24}.
   \end{split}
   \end{equation}
Then, each term of \eqref{eq2.11} will be estimated. First, Lemma \ref{lemma2.3} and \eqref{eq2.7} lead to
   \begin{equation}\label{eq2.12}
   \begin{split}
       \left\| R_{21} \right\|_2   & \leq  C\left\| Au_0\right\| +  C\int_0^t \left\|E(s) \big( v_1'(t-s) + v_2'(t-s) \big)\right\| ds \\
       & \leq  C\left( \left\| u_0\right\|_2 + \sum\limits_{j=1}^{m} t^{\alpha_j} +  \int_0^t  \left\| v_2'(s) \right\|ds  \right),
   \end{split}
   \end{equation}
besides, Lemmas \ref{lemma2.1}-\ref{lemma2.3} and assumption \eqref{eq1.2} give
   \begin{equation}\label{eq2.13}
   \begin{split}
      \left\| R_{22} \right\|_2   & \leq \sum\limits_{j=1}^{m} \left\| \int_0^t AE(t-s)\omega_{\alpha_j}(s) (B_jA^{-1}) Au_0  ds  \right\| \\
      & \leq C  \sum\limits_{j=1}^{m} t^{\alpha_j}  \left\| Au_0\right\| \leq C t^{\alpha-1}\left\| u_0\right\|_2,
   \end{split}
   \end{equation}
and
   \begin{equation}\label{eq2.14}
   \begin{split}
        \left\| R_{23} \right\|_2   & \leq \sum\limits_{j=1}^{m} \left\| \int_0^t AE(t-s) \int_0^s \omega_{\alpha_j}(s-\sigma) (B_jA^{-1}) Av_2'(\sigma)  ds  \right\| \\
      & \leq C  \sum\limits_{j=1}^{m}   \int_0^t \omega_{\alpha_j}(t-\sigma) \left\|  Av_2'(\sigma)  \right\|d\sigma
       \leq C \int_0^t \omega_{\alpha}(t-\sigma) \left\|  Av_2'(\sigma)  \right\|d\sigma.
   \end{split}
   \end{equation}
Then for the estimate of $\left\| R_{24} \right\|_2$, we first utilize \eqref{eq2.6} to obtain
   \begin{equation}\label{eq2.15}
   \begin{split}
       Av_1'(t) = E'(t)(Au_0 - g(0))  + \int_0^t AE(t-s)g'(s)ds.
   \end{split}
   \end{equation}
Hence, we have
   \begin{equation}\label{eq2.16}
   \begin{split}
      AR_{24} & = -\sum\limits_{j=1}^{m} \int_0^t AE(t-s)  (B_jA^{-1}) G_j(s)ds  \\
       & - \sum\limits_{j=1}^{m}\int_0^t AE(t-s) \int_0^s \omega_{\alpha_j}(s-\sigma) H_j(\sigma) d\sigma ds,
   \end{split}
   \end{equation}
from which, $H_j(t)=\int_0^t (B_jA^{-1})AE(t-\theta) g'(\theta)d\theta$ and
   \begin{equation}\label{eq2.17}
   \begin{split}
           G_j(t) = \int_0^t \omega_{\alpha_j}(s)AE(t-s)( Au_0 - g(0) )ds.
   \end{split}
   \end{equation}
Then, Lemma \ref{lemma2.1} can be applied to \eqref{eq2.16} by satisfying two conditions. Firstly, we have $\| G_j(t)\|\leq Ct^{\alpha_j-1}\|Au_0 - g(0)\|\leq Ct^{\alpha_j-1}$. Further we estimate the derivative of $G_j(t)$. By writing
   \begin{equation*}
   \begin{split}
           G_j(t) = \left(\int_0^{t/2}+\int_{t/2}^{t} \right) \omega_{\alpha_j}(s)AE(t-s)( Au_0 - g(0) )ds:= (G_1)_j(t) + (G_2)_j(t).
   \end{split}
   \end{equation*}
Next we employ $\frac{\partial [AE(t-s)]}{\partial t}=-A^2E(t-s)$ to get
   \begin{equation*}
   \begin{split}
          (G_1)_j'(t) = \omega_{\alpha_j}(t/2)AE(t/2)( Au_0 - g(0) ) - \int_0^{t/2} \omega_{\alpha_j}(s)A^2E(t-s)( Au_0 - g(0) )ds,
   \end{split}
   \end{equation*}
thus we have
   \begin{equation}\label{eq2.18}
   \begin{split}
          \left\|(G_1)_j'(t)\right\| = C t^{\alpha_j-2} \|Au_0 - g(0)\| \leq Ct^{\alpha_j-2}.
   \end{split}
   \end{equation}
Further, by applying $\frac{\partial E(t-s)}{\partial s}=AE(t-s)$ and integrating by parts, we obtain
   \begin{equation*}
   \begin{split}
           (G_2)_j(t) = \omega_{\alpha_j}(s)E(t-s)\Big|_{t/2}^{t} ( Au_0 - g(0) )ds - \int_{t/2}^{t} \omega_{\alpha_j}'(s)E(t-s)( Au_0 - g(0) )ds,
   \end{split}
   \end{equation*}
therefore \eqref{eq2.7} gives $\|(G_2)_j(t)\| \leq Ct^{\alpha_j-1}$. Then, Differentiating $(G_2)_j(t)$ and similar to the estimate of $\left\|(G_1)_j'(t)\right\|$, we can yield that
   \begin{equation}\label{eq2.19}
   \begin{split}
          \left\|(G_2)_j'(t)\right\| = C t^{\alpha_j-2} \|Au_0 - g(0)\| \leq Ct^{\alpha_j-2}.
   \end{split}
   \end{equation}
In view of \eqref{eq2.18} and \eqref{eq2.19}, we use Lemmas \ref{lemma2.1} and \ref{lemma2.2} to deduce
   \begin{equation}\label{eq2.20}
   \begin{split}
      \left\|R_{24}\right\|_2 & \leq \sum\limits_{j=1}^{m}t^{\alpha_j-1}
       + C \sum\limits_{j=1}^{m}\int_0^t (t-s)^{\alpha_j-1} \|H_j(s)\|  ds \\
       & \leq C \sum\limits_{j=1}^{m}t^{\alpha_j-1} + \sum\limits_{j=1}^{m}\int_0^t (t-s)^{\alpha_j-1}s^{\alpha-1}  ds \\
       & \leq C\left( t^{\alpha-1} + \sum_{j=1}^{m}t^{\alpha_j+\alpha-1} \right) \leq C t^{\alpha-1}.
   \end{split}
   \end{equation}
Using \eqref{eq2.11}-\eqref{eq2.14} and \eqref{eq2.20}, then
   \begin{equation}\label{eq2.21}
   \begin{split}
       \left\| Av_2'(t)\right\| \leq C \left( t^{\alpha-1}\left\| u_0\right\|_2 +  \int_0^t \omega_{\alpha}(t-\sigma) \left\|  Av_2'(\sigma)  \right\|d\sigma \right),
   \end{split}
   \end{equation}
from which using Gr\"{o}nwall's lemma (see \cite[Lemma 1]{Chen}), we get $\left\| Av_2'(t)\right\| \leq Ct^{\alpha-1}\left\| u_0\right\|_2$.
From \eqref{eq2.10} and above analyses, we have
   \begin{equation}\label{eq2.22}
   \begin{split}
       \left\| v_2''(t)\right\| & \leq C  \left\| Av_2'(t)\right\| + \kappa \left\| v'(t)\right\| +  \sum\limits_{j=1}^{m}\omega_{\alpha_j}(t)\|u_0\|_2 \\
       & +  \sum\limits_{j=1}^{m}\int_0^t \omega_{\alpha_j}(t-s) \|(B_jA^{-1})Av'(s)\| ds \\
       & \leq C t^{\alpha-1}\left\| u_0\right\|_2,
   \end{split}
   \end{equation}
which combines Lemma \ref{lemma2.3}, we finish the proof.
\end{proof}

\vskip 0.05in
According to the relation between $v(t)$ and $u(t)$, we give the following theorems.
\begin{theorem}\label{theorem2.6}
Under the assumptions of Theorem \ref{theorem2.4}. Then the solution of \eqref{eq1.1} satisfies that
   \begin{equation*}
   \begin{split}
            \|u'(t)\| + \|Au(t)\| \leq C e^{-\kappa t}\|u_0\|, \quad 0< t \leq T.
   \end{split}
   \end{equation*}
   \end{theorem}
\begin{proof} By Theorem \ref{theorem2.4} and $v(t)=e^{\kappa t}u(t)$, we have
   \begin{equation*}
   \begin{split}
            \|u'(t)+\kappa u(t)\| + \|Au(t)\| \leq C \|u_0\| e^{-\kappa t},
   \end{split}
   \end{equation*}
which follows from the triangle inequality that
   \begin{equation*}
   \begin{split}
            \|u'(t)\| + \|Au(t)\| \leq e^{-\kappa t} \Big(C\|u_0\|  + \kappa \|v(t)\| \Big),
   \end{split}
   \end{equation*}
which yields the desired result.
\end{proof}

\begin{theorem}\label{theorem2.7}
Under the assumptions of Theorem \ref{theorem2.5}, then for $0< t \leq T$, the solution of \eqref{eq1.1} satisfies
   \begin{equation*}
   \begin{split}
            \|u''(t)\| + \|Au'(t)\| \leq C e^{-\kappa t}t^{\alpha-1}\|u_0\|_2,  \quad \alpha \in \min\limits_{1\leq j\leq m}\{\alpha_j\}.
   \end{split}
   \end{equation*}
   \end{theorem}
\begin{proof} By Theorem \ref{theorem2.5}, Theorem \ref{theorem2.6} and the triangle inequality, we get
   \begin{equation*}
   \begin{split}
            \|Au'(t)\|  \leq  e^{-\kappa t}\|Av'(t)\| +  \kappa\|Au(t)\| \leq Ce^{-\kappa t}t^{\alpha-1}\|u_0\|_2,
   \end{split}
   \end{equation*}
and
   \begin{equation*}
   \begin{split}
            \|u''(t)\|   \leq e^{-\kappa t} \|v''(t)\| +  \kappa^2 e^{-\kappa t} \|v(t)\| + 2\kappa\|u'(t)\| \leq C e^{-\kappa t}t^{\alpha-1}\|u_0\|_2.
   \end{split}
   \end{equation*}
We complete the proof.
\end{proof}

\vskip 0.1in
Moreover, to show the regularity result \eqref{eq2.3}, we need to derive that $\|Av''(t)\|\leq Ct^{\alpha-2}$, which can be proved by the procedure that was utilized to present that $\|Av'(t)\|\leq Ct^{\alpha-1}$, based on certain changes and appropriate assumptions. These naturally deduce that $\|Au''(t)\|\leq Ce^{-\kappa t}t^{\alpha-2}$ from previous analyses.

\section{Establishment of numerical scheme} \label{sec3} Before formulating the time-discrete scheme, we first introduce some helpful notations.

  \subsection{Some notations}

    \vskip 0.2mm
  In this section, we shall give some symbols in order to construct our method. First, we present the temporal levels $0=t_0<t_1 < t_2 < \cdots$ and define that $k_n = t_n - t_{n-1}$ and $t_{n-1/2} = \frac{1}{2} (  t_{n} + t_{n-1} )$ for $n\geq 1$. Moreover, we define that $\mathcal{V}_k=\{\mathcal{V}^n| 0\leq n\leq N\}$ and that
 \begin{equation*}
 \begin{array}{ll}
    \delta_t \mathcal{V}^{n-\frac{1}{2}} =\frac{1}{k_n}(\mathcal{V}^n-\mathcal{V}^{n-1}), \quad   \mathcal{V}^{n-\frac{1}{2}}=\frac{1}{2}(\mathcal{V}^n+\mathcal{V}^{n-1}),  \quad  1\leq n \leq N,
 \end{array}
  \end{equation*}
 and we also define the piecewise constant approximation as follows
      \begin{equation}\label{eq3.1}
      \overline{\mathcal{V}}(t)=
     \begin{cases}
          \mathcal{V}^1, &   t_0 < t < t_1, \\
          \mathcal{V}^{n-1/2}, &   t_{n-1} < t < t_n, \quad n\geq 2.
     \end{cases}
     \end{equation}
Then in view of above results, for approximating the integral term given in \eqref{eq2.2}, we utilize the PI rule to denote the discrete fractional integral (see \cite{McLean})
     \begin{equation}\label{eq3.2}
     \begin{split}
           I^{(\alpha_j)}\mathcal{V}^{n-1/2} &= \frac{1}{k_n}\int_{t_{n-1}}^{t_n}I^{(\alpha_j)}\overline{\mathcal{V}}(t)dt = \frac{1}{k_n}\int_{t_{n-1}}^{t_n}\int_{0}^{t}\omega_{\alpha_j}(t-\vartheta)\overline{\mathcal{V}}(\vartheta)d\vartheta dt \\
           & = w_{n1}^{(j)}\mathcal{V}^1k_1 + \sum\limits_{p=2}^{n}w_{np}^{(j)} \mathcal{V}^{p-1/2}k_p, \quad 1\leq j \leq m,
     \end{split}
     \end{equation}
where
     \begin{equation}\label{eq3.3}
     \begin{split}
           w_{np}^{(j)} = \frac{1}{k_nk_p}\int_{t_{n-1}}^{t_n}\int_{t_{p-1}}^{\min\{t,t_p\}}\omega_{\alpha_j}(t-\vartheta)  d\vartheta dt >0.
     \end{split}
     \end{equation}
Further, for $1\leq p\leq n-1$, we obtain
     \begin{equation}\label{eq3.4}
     \begin{split}
           w_{np}^{(j)} = \frac{1}{k_nk_p}\int_{t_{n-1}}^{t_n}\int_{t_{p-1}}^{t_p}\omega_{\alpha_j}(t-\vartheta)  d\vartheta dt =  \frac{\lambda_{n,p}^{(j)}-\lambda_{n-1,p}^{(j)}}{k_n k_p \Gamma(\alpha_j+2)},
     \end{split}
     \end{equation}
from which, $\lambda_{n,p}^{(j)} =  (t_{n}-t_{p-1})^{\alpha_j+1} - (t_{n}-t_{p})^{\alpha_j+1}$, and for $n \geq 1$, then
     \begin{equation}\label{eq3.5}
     \begin{split}
           w_{nn}^{(j)} = \frac{1}{k_n^2}\int_{t_{n-1}}^{t_n}\int_{t_{n-1}}^{t}\omega_{\alpha_j}(t-\vartheta)  d\vartheta dt = \frac{k_n^{\alpha_j-1}}{\Gamma(\alpha_j+2)}.
     \end{split}
     \end{equation}

    \vskip 0.2mm
Besides, for the inhomogeneous term in \eqref{eq2.2}, we denote
     \begin{equation}\label{eq3.6}
     \begin{split}
           g^{n-\frac{1}{2}}\approx \frac{1}{k_n}\int_{t_{n-1}}^{t_n} g(t) dt, \qquad n\geq 1,
     \end{split}
     \end{equation}
from which, we suppose that
     \begin{equation}\label{eq3.7}
     \begin{split}
            \left\|  g^{1/2}k_1 - \int_{t_0}^{t_1} g(t)dt  \right\| \leq C \int_{t_0}^{t_1} t \|g'(t)\| dt,
     \end{split}
     \end{equation}
and that
     \begin{equation}\label{eq3.8}
     \begin{split}
            \left\|  g^{n-1/2}k_n - \int_{t_{n-1}}^{t_n} g(t)dt  \right\| \leq C k_n^2 \int_{t_{n-1}}^{t_n}  \|g''(t)\| dt, \quad n\geq 2.
     \end{split}
     \end{equation}
Also, e.g., $g^{n-\frac{1}{2}}= \frac{1}{2}\left[ g(t_{n-1}) + g(t_{n}) \right]$ or $g^{n-\frac{1}{2}}=g(t_{n-1/2})$ are permissible.

    \vskip 2mm
Then, to overcome the singular behaviour of the exact solution at $t=0$, we give some hypotheses of non-uniform meshes \cite{McLean} that
     \begin{equation}\label{eq3.9}
     \begin{split}
            k_n\leq C_{\gamma}k \min\left\{1,t_n^{1-1/\gamma}\right\}, \quad n\geq 1, \quad \gamma\geq 1,
     \end{split}
     \end{equation}
where $C_{\gamma}$ does not depend on $k$, and that
      \begin{equation}\label{eq3.10}
     \begin{split}
                 t_1=k_1\geq c_{\gamma}k^{\gamma}, \quad t_n \leq C_{\gamma}t_{n-1},    \quad n\geq 2,
     \end{split}
     \end{equation}
and that (a more rigid hypothesis)
      \begin{equation}\label{eq3.11}
     \begin{split}
          0\leq k_{n+1} - k_n\leq C_{\gamma}k^2 \min\left\{1,t_n^{1-2/\gamma}\right\}, \quad n\geq 2.
     \end{split}
     \end{equation}
Thence, for $0\leq t\leq T< \infty$, the case satisfying the hypotheses \eqref{eq3.9}-\eqref{eq3.11} is
      \begin{equation}\label{eq3.12}
     \begin{split}
          t_n = (nk)^{\gamma}, \quad  k= T^{1/\gamma}/N,   \quad 0\leq n \leq N.
     \end{split}
     \end{equation}

 \subsection{Time discrete scheme}

 First, we integrate \eqref{eq2.2} from $t=t_{n-1}$ to $t=t_n$ and multiply the term $\frac{1}{k_n}$, then
    \begin{equation}\label{eq3.13}
   \begin{split}
      \delta_tv^{n-\frac{1}{2}} + \frac{1}{k_n}\int_{t_{n-1}}^{t_n}  Av(t)dt &+  \frac{1}{k_n}\sum\limits_{j=1}^m\int_{t_{n-1}}^{t_n}(\omega_{\alpha_j}*B_jv)(t)dt \\
      & -  \frac{\kappa}{k_n}\int_{t_{n-1}}^{t_n}v(t)dt  = \frac{1}{k_n} \int_{t_{n-1}}^{t_n}g(t)dt,
   \end{split}
   \end{equation}
which follows from \eqref{eq3.1}-\eqref{eq3.8} that
    \begin{equation}\label{eq3.14}
   \begin{split}
      \delta_tv^{\frac{1}{2}} +  Av^{1} &+  \sum\limits_{j=1}^m I^{(\alpha_j)} B_jv^{\frac{1}{2}}
       -  \kappa v^{1}  = g^{\frac{1}{2}} + \mathcal{R}^{1/2},
   \end{split}
   \end{equation}
and for $n\geq 2$ that
    \begin{equation}\label{eq3.15}
   \begin{split}
      \delta_tv^{n-\frac{1}{2}} +  Av^{n-\frac{1}{2}} &+  \sum\limits_{j=1}^m I^{(\alpha_j)} B_jv^{n-\frac{1}{2}}
       -  \kappa v^{n-\frac{1}{2}}  = g^{n-\frac{1}{2}} + \mathcal{R}^{n-1/2},
   \end{split}
   \end{equation}
where $v^n=v(t_n)$, $v^{n-\frac{1}{2}}=\frac{v^n+v^{n-1}}{2}$ and $\mathcal{R}^{n-1/2}=\sum\limits_{s=1}^{4}\mathcal{R}_s^{n-1/2}$ for $n\geq 1$, and
    \begin{equation}\label{eq3.16}
   \begin{split}
      & \mathcal{R}_1^{1/2} =  Av^{1} - \frac{1}{k_1}\int_{t_{0}}^{t_1}  Av(t)dt, \\
      & \mathcal{R}_1^{n-1/2} =  Av^{n-\frac{1}{2}} - \frac{1}{k_n}\int_{t_{n-1}}^{t_n}  Av(t)dt, \quad n\geq 2, \\
      & \mathcal{R}_2^{n-1/2} = \sum\limits_{j=1}^m  \left( I^{(\alpha_j)} B_jv^{n-\frac{1}{2}} - \frac{1}{k_n}\int_{t_{n-1}}^{t_n}(\omega_{\alpha_j}*B_jv)(t)dt \right),  \quad n,m \geq 1, \\
      & \mathcal{R}_3^{1/2} = -  \kappa\left( v^{1} - \frac{1}{k_1}\int_{t_{0}}^{t_1}  v(t)dt\right), \\
      & \mathcal{R}_3^{n-1/2} = -  \kappa\left( v^{n-\frac{1}{2}} - \frac{1}{k_n}\int_{t_{n-1}}^{t_n}  v(t)dt\right), \quad n\geq 2, \\
      & \mathcal{R}_4^{n-1/2} = g^{n-\frac{1}{2}} - \frac{1}{k_n}\int_{t_{n-1}}^{t_n}  g(t)dt, \quad n\geq 1.
   \end{split}
   \end{equation}
Then omitting the truncation errors $\mathcal{R}^{n-1/2}$ and replacing $v^n$ with its numerical approximation $V^n$, then time discrete scheme is yielded as follows
    \begin{equation}\label{eq3.17}
   \begin{split}
      \delta_tV^{\frac{1}{2}} +  AV^{1} &+  \sum\limits_{j=1}^m I^{(\alpha_j)} B_jV^{\frac{1}{2}}
       -  \kappa V^{1}  = g^{\frac{1}{2}},
   \end{split}
   \end{equation}
    \begin{equation}\label{eq3.18}
   \begin{split}
      \delta_tV^{n-\frac{1}{2}} +  AV^{n-\frac{1}{2}} &+  \sum\limits_{j=1}^m I^{(\alpha_j)} B_jV^{n-\frac{1}{2}}
       -  \kappa V^{n-\frac{1}{2}}  = g^{n-\frac{1}{2}}, \quad 2\leq n \leq N,
   \end{split}
   \end{equation}
    \begin{equation}\label{eq3.19}
   \begin{split}
      V^0 = u_0.
   \end{split}
   \end{equation}

\section{Analysis of time discrete scheme} \label{sec4} In this section, we shall present the stability and convergence of the proposed time-discrete scheme \eqref{eq3.17}-\eqref{eq3.19}.

\subsection{Stability} Here, we first establish the stability result of the numerical scheme. Then the following theorem is obtained.

\begin{theorem}\label{theorem4.1} Assume that the numerical solution $V^n$ is denoted by \eqref{eq3.17} and \eqref{eq3.18} for $1\leq n \leq N$. Let $U^n$ be the approximate solution of $u(t_n)$ given in \eqref{eq1.1}. Then for $T<\infty$ and $0\leq \kappa <\infty$, it holds that
   \begin{equation*}
   \begin{split}
      \max\limits_{1\leq n \leq N} \|V^n\| \leq C(T) \left( \|V^0\| + \sum\limits_{n=1}^{N} k_n \left\|g^{n-1/2}\right\| \right),
   \end{split}
   \end{equation*}
and that
   \begin{equation*}
   \begin{split}
      \max\limits_{1\leq n \leq N} \|U^n\| \leq C(T) \left( \|U^0\| + \sum\limits_{n=1}^{N} k_n e^{\kappa t_{n-1/2}} \left\|f^{n-1/2}\right\| \right).
   \end{split}
   \end{equation*}
\end{theorem}
\begin{proof}
  By taking the inner product of \eqref{eq3.17}-\eqref{eq3.18} with $k_1V^{1}$ and $k_nV^{n-\frac{1}{2}}$ correspondingly, and summing \eqref{eq3.8} for $n$ from 2 to $N$, we obtain that
   \begin{equation}\label{eq4.1}
   \begin{split}
      k_1(\delta_tV^{\frac{1}{2}}, V^{1})  & + k_1 (AV^{1}, V^{1}) + k_1 \sum\limits_{j=1}^{m} \left(I^{(\alpha_j)}B_jV^{1/2}, V^{1}\right) \\
      & - \kappa k_1 (V^1,V^{1})  = k_1\left(g^{\frac{1}{2}},V^{1}\right),
   \end{split}
   \end{equation}
  and that
   \begin{equation}\label{eq4.2}
   \begin{split}
     & \sum\limits_{n=2}^{N}k_n(\delta_tV^{n-\frac{1}{2}}, V^{n-\frac{1}{2}})   + \sum\limits_{n=2}^{N}k_n (AV^{n-\frac{1}{2}}, V^{n-\frac{1}{2}})- \kappa \sum\limits_{n=2}^{N}k_n (V^{n-\frac{1}{2}},V^{n-\frac{1}{2}})  \\
      &  +  \sum\limits_{j=1}^{m}\sum\limits_{n=2}^{N} k_n \left(I^{(\alpha_j)}B_jV^{n-1/2}, V^{n-\frac{1}{2}}\right) = \sum\limits_{n=2}^{N}k_n \left(g^{n-\frac{1}{2}},V^{n-\frac{1}{2}}\right).
   \end{split}
   \end{equation}
Adding above two formulae and applying the positiveness of operator $A$, we have
   \begin{equation}\label{eq4.3}
   \begin{split}
       k_1(\delta_tV^{\frac{1}{2}}, V^{1}) & + \sum\limits_{n=2}^{N}k_n(\delta_tV^{n-\frac{1}{2}}, V^{n-\frac{1}{2}})  - \kappa \left( k_1 \|V^1\|^2 + \sum\limits_{n=2}^{N}k_n \|V^{n-\frac{1}{2}}\|^2  \right)  \\
      &   +  \sum\limits_{j=1}^{m} \left\{k_1 \left(I^{(\alpha_j)}B_jV^{\frac{1}{2}}, V^{1}\right) + \sum\limits_{n=2}^{N} k_n \left(I^{(\alpha_j)}B_jV^{n-\frac{1}{2}}, V^{n-\frac{1}{2}}\right) \right\} \\
      & =  k_1\left(g^{\frac{1}{2}},V^{1}\right) + \sum\limits_{n=2}^{N}k_n \left(g^{n-\frac{1}{2}},V^{n-\frac{1}{2}}\right).
   \end{split}
   \end{equation}
Let $B_j$ be the positive-definite operator and then from \cite[Lemma 3.1]{Mustapha}, we get
   \begin{equation}\label{eq4.4}
   \begin{split}
       k_1 \left(I^{(\alpha_j)}B_jV^{\frac{1}{2}}, V^{1}\right) + \sum\limits_{n=2}^{N} k_n \left(I^{(\alpha_j)}B_jV^{n-\frac{1}{2}}, V^{n-\frac{1}{2}}\right)
       \geq 0.
   \end{split}
   \end{equation}
Therefore, using the Cauchy-Schwarz inequality, then \eqref{eq4.3} turns into
   \begin{equation}\label{eq4.5}
   \begin{split}
       k_1\left(\delta_tV^{\frac{1}{2}}, V^{1}\right) & + \sum\limits_{n=2}^{N}k_n\left(\delta_tV^{n-\frac{1}{2}}, V^{n-\frac{1}{2}}\right)  - \kappa \left( k_1 \|V^1\|^2 + \sum\limits_{n=2}^{N}k_n \|V^{n-\frac{1}{2}}\|^2  \right)  \\
      & \leq  k_1 \left\|g^{\frac{1}{2}}\right\| \left\|V^{1}\right\| + \sum\limits_{n=2}^{N}k_n \left\|g^{n-\frac{1}{2}}\right\| \left\|V^{n-\frac{1}{2}}\right\|.
   \end{split}
   \end{equation}
Then it is easy to yield
   \begin{equation}\label{eq4.6}
   \begin{split}
       \left(\delta_tV^{\frac{1}{2}}, V^{1}\right)\geq \frac{\left\|V^{1}\right\|^2- \left\|V^{0}\right\|^2}{2k_1},  \quad \left(\delta_tV^{n-\frac{1}{2}}, V^{n-\frac{1}{2}}\right) =\frac{\left\|V^{n}\right\|^2- \left\|V^{n-1}\right\|^2}{2k_n}.
   \end{split}
   \end{equation}
Thus, we further obtain
   \begin{equation}\label{eq4.7}
   \begin{split}
      \frac{\left\|V^{N}\right\|^2- \left\|V^{0}\right\|^2}{2} & \leq k_1 \left( \kappa\left\|V^{1}\right\| + \left\|g^{1/2}\right\| \right)
       \left\|V^{1}\right\| \\
       & + \sum\limits_{n=2}^{N}k_n \left( \kappa \left\|V^{n-\frac{1}{2}}\right\| + \left\|g^{n-\frac{1}{2}}\right\| \right) \left\|V^{n-\frac{1}{2}}\right\|.
   \end{split}
   \end{equation}
Next, choosing a suitable $\mathcal{M}$ such that $\left\|V^{\mathcal{M}}\right\|=\max\limits_{0\leq n \leq N} \left\|V^{n}\right\|$, then
   \begin{equation}\label{eq4.8}
   \begin{split}
      \left\|V^{\mathcal{M}}\right\|^2  & \leq  \left\|V^{0}\right\|^2 + 2k_1 \left( \kappa\left\|V^{1}\right\| + \left\|g^{1/2}\right\| \right)
       \left\|V^{1}\right\| \\
       & + 2\sum\limits_{n=2}^{\mathcal{M}}k_n \left( \kappa \left\|V^{n-\frac{1}{2}}\right\| + \left\|g^{n-\frac{1}{2}}\right\| \right) \left\|V^{n-\frac{1}{2}}\right\| \\
       & \leq  \left\|V^{0}\right\|\left\|V^{\mathcal{M}}\right\| + 2k_1 \left( \kappa\left\|V^{1}\right\| + \left\|g^{1/2}\right\| \right)
      \left\|V^{\mathcal{M}}\right\| \\
       & + 2\sum\limits_{n=2}^{\mathcal{M}}k_n \left( \kappa \left\|V^{n-\frac{1}{2}}\right\| + \left\|g^{n-\frac{1}{2}}\right\| \right) \left\|V^{\mathcal{M}}\right\|.
   \end{split}
   \end{equation}
Consequently,
   \begin{equation}\label{eq4.9}
   \begin{split}
      \left\|V^{N}\right\| \leq \left\|V^{\mathcal{M}}\right\|
       & \leq  \left\|V^{0}\right\| + 2k_1 \left( \kappa\left\|V^{1}\right\| + \left\|g^{1/2}\right\| \right)
       \\
       & + 2\sum\limits_{n=2}^{N}k_n \left( \kappa \left\|V^{n-\frac{1}{2}}\right\| + \left\|g^{n-\frac{1}{2}}\right\| \right) \\
       & \leq  \kappa k_N \left\|V^{N}\right\| + 2\kappa \sum\limits_{n=1}^{N-1}( k_n+k_{n+1} ) \left\|V^{n}\right\| \\
       & + \left( \left\|V^{0}\right\| + 2\sum\limits_{n=1}^{N}k_n \left\|g^{n-\frac{1}{2}}\right\|  \right).
   \end{split}
   \end{equation}
Then an application of Gr\"{o}nwall's lemma (see \cite{Sloan}) yields
   \begin{equation}\label{eq4.10}
   \begin{split}
      \left\|V^{N}\right\| & \leq C \sum\limits_{n=1}^{N-1}\left( k_n+k_{n+1} \right) \left( \left\|V^{0}\right\| + 2\sum\limits_{n=1}^{N}k_n \left\|g^{n-\frac{1}{2}}\right\|  \right) \\
      & \leq C(T) \left( \left\|V^{0}\right\| + \sum\limits_{n=1}^{N}k_n \left\|g^{n-\frac{1}{2}}\right\|  \right),
   \end{split}
   \end{equation}
which can finish the proof by using \eqref{eq2.1}.
\end{proof}

\subsection{Convergence} Based on above analyses, we shall deduce the convergence of the scheme by the energy argument. We first define
   \begin{equation}\label{eq4.11}
   \begin{split}
         \rho^n = v(t_n) - V^n,  \quad  e^n = u(t_n) - U^n,  \quad 0 \leq n \leq N.
   \end{split}
   \end{equation}
 Then we subtract \eqref{eq3.17}-\eqref{eq3.18} from \eqref{eq3.14}-\eqref{eq3.15} to get error equations as follows
    \begin{equation}\label{eq4.12}
   \begin{split}
      \delta_t\rho^{\frac{1}{2}} +  A\rho^{1} &+  \sum\limits_{j=1}^m I^{(\alpha_j)} B_j\rho^{\frac{1}{2}}
       -  \kappa \rho^{1}  = \mathcal{R}^{\frac{1}{2}},
   \end{split}
   \end{equation}
    \begin{equation}\label{eq4.13}
   \begin{split}
      \delta_t\rho^{n-\frac{1}{2}} +  A\rho^{n-\frac{1}{2}} &+  \sum\limits_{j=1}^m I^{(\alpha_j)} B_j\rho^{n-\frac{1}{2}}
       -  \kappa \rho^{n-\frac{1}{2}}  = \mathcal{R}^{n-\frac{1}{2}}
   \end{split}
   \end{equation}
with $2\leq n \leq N$ and $\rho^0=0$, where $\mathcal{R}^{n-\frac{1}{2}}$ is given in \eqref{eq3.16}.

\vskip 2mm
In order to further derive the convergence, we shall introduce several auxiliary lemmas. At first, from \cite[Corollary 3.4]{McLean}, we can yield the following two lemmas.

\begin{lemma}\label{lemma4.3} Supposing that $t_n$ satisfies the assumptions \eqref{eq3.9}-\eqref{eq3.11} and the exact solution satisfies the regularity condition \eqref{eq2.3}, then for $1\leq m<\infty$, $\gamma\geq 1$ and $\alpha=\min\limits_{1 \leq j \leq m}\{\alpha_j\}$, we have
      \begin{equation*}
     \begin{split}
       \sum\limits_{n=1}^N k_n \left\|\mathcal{R}_2^{n-\frac{1}{2}}\right\| \leq C_{\alpha,\gamma,T,m} \times \Xi(k,\alpha,\gamma,T),
     \end{split}
     \end{equation*}
where
      \begin{equation*}
     \begin{split}
       \Xi(k,\alpha,\gamma,T):=
        \begin{cases}
          k^{\gamma(\alpha+1)}, & \mbox{if } 1\leq \gamma < 2/(\alpha+1), \\
          k^{2}\log(t_N/t_1), & \mbox{if }\gamma = 2/(\alpha+1), \\
          k^2, & \mbox{if } \gamma > 2/(\alpha+1).
        \end{cases}
     \end{split}
     \end{equation*}
\end{lemma}

\begin{lemma}\label{lemma4.4} Assuming that $t_n$ satisfies the assumptions \eqref{eq3.9}-\eqref{eq3.11} and the modified source term satisfies $t\|g'(t)\|+t^2\|g''(t)\|\leq Ct^{\alpha}$, then for $\alpha=\min\limits_{1 \leq j \leq m}\{\alpha_j\}$ and $\gamma\geq 1$, it holds that
      \begin{equation*}
     \begin{split}
       \sum\limits_{n=1}^N k_n \left\|\mathcal{R}_4^{n-\frac{1}{2}}\right\| \leq C_{\alpha,\gamma} \times
        \Xi(k,\alpha,\gamma,T).
     \end{split}
     \end{equation*}
\end{lemma}

Further, we will estimate the remaining error terms in \eqref{eq3.16} based on certain reasonable conditions. Firstly, we give the following result.
\begin{lemma}\label{lemma4.5} Assume that assumptions \eqref{eq3.9}-\eqref{eq3.11} hold and the exact solution satisfies the regularity \eqref{eq2.3}. Then for $\gamma\geq 1$ and $\alpha=\min\limits_{1 \leq j \leq m}\{\alpha_j\}$, it holds that
      \begin{equation*}
     \begin{split}
       \sum\limits_{n=1}^N k_n \left\|\mathcal{R}_1^{n-\frac{1}{2}}\right\| \leq C_{\alpha,\gamma} \times \Xi(k,\alpha,\gamma,T).
     \end{split}
     \end{equation*}
\end{lemma}
\begin{proof} First, when $n=1$, we use the Taylor expansion with integral remainder to obtain
    \begin{equation}\label{eq4.14}
   \begin{split}
      k_1\mathcal{R}_1^{\frac{1}{2}} = \int_{0}^{k_1}\int_{t}^{k_1} Av'(\zeta)d\zeta dt = \int_{0}^{k_1}\int_{0}^{\zeta} Av'(\zeta)dt d\zeta = \int_{0}^{k_1}\zeta Av'(\zeta) d\zeta.
   \end{split}
   \end{equation}
Similarly, we also have
      \begin{equation}\label{eq4.15}
     \begin{split}
        v(t)-v(t_{n-1/2}) = (t-t_{n-1/2})v'(t_{n-1/2}) + \int_{t_{n-1/2}}^{t} (t-\zeta)v''(\zeta)d\zeta,
     \end{split}
     \end{equation}
and
      \begin{equation}\label{eq4.16}
     \begin{split}
        v^{n-\frac{1}{2}}-v(t_{n-\frac{1}{2}}) = \frac{1}{2} \left[ \int_{t_{n-1}}^{t_{n-\frac{1}{2}}} (\zeta-t_{n-1})v''(\zeta)d\zeta - \int_{t_{n-\frac{1}{2}}}^{t_{n}} (\zeta-t_{n})v''(\zeta)d\zeta \right].
     \end{split}
     \end{equation}
Then for $n\geq2$, we yield
    \begin{equation}\label{eq4.17}
   \begin{split}
      k_n\mathcal{R}_1^{n-\frac{1}{2}} &= \int_{t_{n-1}}^{t_n} A\left[ v^{n-\frac{1}{2}}-v(t_{n-1/2})  \right] d\zeta + \int_{t_{n-1}}^{t_n} A\left[ v(t_{n-1/2}) -v(t)  \right] d\zeta \\
      & = \frac{k_n}{2} \left[ \int_{t_{n-1}}^{t_{n-1/2}} (\zeta-t_{n-1})Av''(\zeta)d\zeta - \int_{t_{n-1/2}}^{t_{n}} (\zeta-t_{n})Av''(\zeta)d\zeta \right] \\
      & - \frac{Av'(t_{n-1/2})}{2}(t-t_{n-1/2})^2 \Big|_{t_{n-1}}^{t_n} - \int_{t_{n-1}}^{t_n}\int_{t_{n-1/2}}^{t} (t-\zeta)Av''(\zeta)d\zeta dt.
   \end{split}
   \end{equation}
Thus, using \eqref{eq4.14}, \eqref{eq4.17}, regularity condition \eqref{eq2.3} and \eqref{eq3.10}, we obtain
      \begin{equation*}
     \begin{split}
       \sum\limits_{n=1}^N k_n \left\|\mathcal{R}_1^{n-\frac{1}{2}}\right\| &\leq  \int_{0}^{k_1}\zeta \|Av'(\zeta)\| d\zeta + \sum\limits_{n=2}^N k_n^2\int_{t_{n-1}}^{t_n} \|Av''(\zeta)\|d\zeta \\
       & \leq C_{\alpha}\left( \int_{0}^{k_1}\zeta^{\alpha}d\zeta + \sum\limits_{n=2}^N k_n^3t_n^{\alpha-2} \right) \\
       & \leq C_{\alpha,\gamma}\left( k^{\gamma(\alpha+1)} + k^2 \sum\limits_{n=2}^N t_n^{\alpha-2/\gamma}k_n \right),
     \end{split}
     \end{equation*}
from which,
      \begin{equation}\label{eq4.18}
     \begin{split}
      \sum\limits_{n=2}^N t_n^{\alpha-2/\gamma}k_n \leq C\int_{k_1}^{t_N}\zeta^{\alpha-2/\gamma}d\zeta \leq C\times
        \begin{cases}
          \frac{k^{\gamma(\alpha+1)-2}}{2/\gamma-(\alpha+1)}, & \mbox{if } \gamma < \frac{2}{\alpha+1}, \\
           \log(t_N/t_1), & \mbox{if }\gamma = \frac{2}{\alpha+1}, \\
          \frac{t_n^{(\alpha+1)-2/\gamma}}{(\alpha+1)-2/\gamma}, & \mbox{if } \gamma > \frac{2}{\alpha+1}.
        \end{cases}
     \end{split}
     \end{equation}
This completes the proof.
\end{proof}

Then, after a similar analysis, the following result holds.
\begin{lemma}\label{lemma4.6} Suppose that assumptions \eqref{eq3.9}-\eqref{eq3.11} are valid and the exact solution satisfies the regularity \eqref{eq2.3}. For $\gamma\geq 1$, $\kappa\geq0$ and $\alpha=\min\limits_{1 \leq j \leq m}\{\alpha_j\}$, it holds that
      \begin{equation*}
     \begin{split}
       \sum\limits_{n=1}^N k_n \left\|\mathcal{R}_3^{n-\frac{1}{2}}\right\| \leq C_{\alpha,\kappa,\gamma,T} \times k^2.
     \end{split}
     \end{equation*}
\end{lemma}
\begin{proof} For $n=1$, using the Taylor expansion with integral remainder, we get
    \begin{equation}\label{eq4.19}
   \begin{split}
      k_1\mathcal{R}_3^{\frac{1}{2}} = -\kappa\int_{0}^{k_1}\int_{t}^{k_1} v'(\zeta)d\zeta dt = -\kappa\int_{0}^{k_1}\zeta v'(\zeta) d\zeta.
   \end{split}
   \end{equation}
Then for $n\geq2$, we employ \eqref{eq4.15} and \eqref{eq4.16} to obtain
    \begin{equation}\label{eq4.20}
   \begin{split}
      k_n\mathcal{R}_3^{n-\frac{1}{2}} &= \int_{t_{n-1}}^{t_n} \left[ v^{n-\frac{1}{2}}-v(t_{n-1/2})  \right] d\zeta + \int_{t_{n-1}}^{t_n} \left[ v(t_{n-1/2}) -v(t)  \right] d\zeta \\
      & =- \frac{\kappa k_n}{2} \left[ \int_{t_{n-1}}^{t_{n-1/2}} (\zeta-t_{n-1})v''(\zeta)d\zeta - \int_{t_{n-1/2}}^{t_{n}} (\zeta-t_{n})v''(\zeta)d\zeta \right] \\
      & + \frac{\kappa v'(t_{n-1/2})}{2}(t-t_{n-1/2})^2 \Big|_{t_{n-1}}^{t_n} + \kappa\int_{t_{n-1}}^{t_n}\int_{t_{n-1/2}}^{t} (t-\zeta)v''(\zeta)d\zeta dt.
   \end{split}
   \end{equation}
Next, applying \eqref{eq4.19}, \eqref{eq4.20}, Theorem \ref{theorem2.4} and Theorem \ref{theorem2.5}, we have
      \begin{equation*}
     \begin{split}
       \sum\limits_{n=1}^N k_n \left\|\mathcal{R}_3^{n-\frac{1}{2}}\right\| &\leq  \kappa\int_{0}^{k_1}\zeta \|v'(\zeta)\| d\zeta + \kappa\sum\limits_{n=2}^N k_n^2\int_{t_{n-1}}^{t_n} \|v''(\zeta)\|d\zeta \\
       & \leq C_{\alpha,\kappa}\left( \int_{0}^{k_1}\zeta d\zeta + \sum\limits_{n=2}^N k_n^2 \left(t_n^{\alpha}- t_{n-1}^{\alpha}\right) \right) \\
       & \leq C_{\alpha,\kappa,\gamma}\left( k_1^2 + k^2 \sum\limits_{n=2}^N \left(t_n^{\alpha}- t_{n-1}^{\alpha}\right) \right),
     \end{split}
     \end{equation*}
which finishes the proof.
\end{proof}

\vskip 0.1in
Based on above analyses, we can yield the following convergence result.
\begin{theorem}\label{theorem4.7} Let $V^n$ denoted by \eqref{eq3.17} and \eqref{eq3.18} and $U^n$ be the approximate solution of and $v(t_n)$ and $u(t_n)$, respectively. For $T<\infty$, $\alpha=\min\limits_{1 \leq j \leq m}\{\alpha_j\}$, $1\leq m <\infty$, $\gamma\geq 1$ and $0\leq \kappa <\infty$, then it holds that
   \begin{equation*}
   \begin{split}
      & \max\limits_{1\leq n \leq N} \|V^n-v(t_n)\| \leq C_{\alpha,\kappa,\gamma,T,m} \times \Xi(k,\alpha,\gamma,T), \\
      &  \max\limits_{1\leq n \leq N} \|U^n-u(t_n)\| \leq C_{\alpha,\kappa,\gamma,T,m} \times \Xi(k,\alpha,\gamma,T),
   \end{split}
   \end{equation*}
where the notation $\Xi(k,\alpha,\gamma,T)$ is denoted in Lemma \ref{lemma4.3}.
\end{theorem}
\begin{proof} In view of \eqref{eq4.11}-\eqref{eq4.13} and Theorem \ref{theorem4.1}, we get
   \begin{equation*}
      \begin{split}
      &\max\limits_{1\leq n \leq N} \|\rho^n\| \leq C(T) \left( \|\rho^0\| + \sum_{n=1}^{N} k_n \left\|\mathcal{R}^{n-1/2}\right\| \right), \\
      &\max\limits_{1\leq n \leq N} \|e^n\| \leq C(T) \left( \|e^0\| + \sum_{n=1}^{N} k_n \left\|\mathcal{R}^{n-1/2}\right\| \right),
      \end{split}
   \end{equation*}
from which $\|\rho^0\|=\|e^0\|=0$. Then utilizing the triangle inequality, \eqref{eq2.1} and Lemmas \ref{lemma4.3}-\ref{lemma4.6}, the proof is completed.
\end{proof}

\section{Numerical experiment} \label{sec5} In this section, in order to further show the time convergence of proposed scheme, we formulate the fully discrete scheme by the time semidiscrete scheme \eqref{eq3.17}-\eqref{eq3.19} and a standard spatial finite difference method, with the spatial step size $h=\frac{L}{M}$ ($M\in \mathbb{Z}^+$ is the number of spatial partitions). Below we choose the parameters $T=L=1$ and denote the $L_{2}$-norm error
     \begin{equation*}
     \mathcal{E}_{CN}(N,M) = \max\limits_{1\leq n \leq N} \|U^n-u(t_n)\|,
     \end{equation*}
and the temporal convergence rate
     \begin{equation*}
       rate_{CN}^k = \log_{2}\left(\frac{\mathcal{E}_{CN}(N,M)}{\mathcal{E}_{CN}(2N,M)}\right).
     \end{equation*}
In addition, we define the following notations for illustrating the numerical stability of the proposed scheme,
     \begin{equation*}
          U^{*} = \max\limits_{1\leq n \leq N} \|U^n\|,  \quad  U^{**} = \sum_{n=1}^{N} k_n \|U^n\|.
     \end{equation*}

\begin{example} \label{Ex1} Here we first consider the one-dimensional case ($d=1$) of \eqref{eq1.1} with the parameter $m=2$. Set the operators $A=B_1=B_2=-\frac{\partial^2}{\partial x_1^2}$ over the domain $\Omega=(0,L)$ with the homogeneous Dirichlet boundary conditions (DBCs). To meet the regularity assumption \eqref{eq2.3}, the exact solution of \eqref{eq1.1} is given via
    \begin{equation*}
              u(x_1,t) =  \left( t^{1+\alpha_1} + t^{1+\alpha_2} \right) e^{-\kappa t}\sin \pi x_1 ,
   \end{equation*}
thus the initial condition $u_0(x_1)=0$ and the source term $f(x_1,t)$ can be computed accordingly.
\end{example}

In Table \ref{eqtable1}, we list the $L_2$-norm errors and time convergence rates by fixing $\kappa=1$, $\alpha_1=0.2$, $\alpha_2=0.8$ and $M=256$, from which,  we discuss three cases regarding grading index $\gamma$, i.e., $\gamma=1$, $\gamma=\frac{2}{\alpha+1}$ and $\gamma=\frac{2}{\alpha+1}+1$, respectively. It can be seen clearly from Table \ref{eqtable1} that time convergence rates of proposed scheme reach the order $1+\min\{\alpha_1,\alpha_2\}$ with uniform temporal step sizes, and second-order accuracy for time can be obtained with nonuniform temporal step sizes ($\gamma\geq\frac{2}{\alpha+1}$). These are consistent with Theorem \ref{theorem4.7}. Below we only consider the cases with the optimal grading index $\gamma=\frac{2}{\alpha+1}$.

Table \ref{eqtable2} shows the $L_2$-norm errors and time convergence rates when $\kappa=2$, $\gamma=\frac{2}{\alpha+1}$ and $M=512$, from which we present three cases with different $\alpha_1$ and $\alpha_2$, including (i) $\alpha_1<\frac{1}{2}<\alpha_2$, (ii) $\alpha_1,\alpha_2<\frac{1}{2}$ and (iii) $\alpha_1,\alpha_2>\frac{1}{2}$. Then, the numerical results in Table \ref{eqtable2} demonstrate that unform temporal second-order accuracy can be yielded under three situations.

Besides, in Table \ref{eqtable3}, fixed $\alpha_1=\alpha_2=0.5$, $\gamma=\frac{2}{\alpha+1}$ and $M=512$, we discuss three cases about different tempered parameter $\kappa$, involving $\kappa=0.2$, $\kappa=1$ and $\kappa=5$, respectively, from which the results approximate second-order convergence for time. These validate our theoretical analysis and illustrate the effectiveness of proposed scheme.

Then, in Table \ref{eqtable4}, fixed $\kappa=1$, $\gamma=\frac{2}{\alpha+1}$ and $M=256$, we list two types of numerical solution $U^{*}$ and $U^{**}$. Table \ref{eqtable4} show that as $N$ increases gradually, the value of $U^{*}$ gradually stabilizes and finally remains unchanged, and the value of $U^{**}$ has maintained a non-increasing trend. These results demonstrate the numerical stability of proposed scheme in the time direction. Furthermore, the same phenomenon is exhibited in Table \ref{eqtable5}, by fixing some different parameters.

\begin{table}
\caption{Example \ref{Ex1}:  $L_2$-norm errors and time convergence rates when $\kappa=1$, $\alpha_1=0.2$, $\alpha_2=0.8$ and $M=256$.}\label{eqtable1}
\begin{center} \footnotesize
\renewcommand\arraystretch{1.5}
\noindent\[
\begin{array}{|c||c|c|c|c|c|c|}
\hline
     & \multicolumn{2}{|c|}{\text{$\gamma=1$}} & \multicolumn{2}{|c|}{\text{$\gamma=\frac{2}{\alpha+1}$}} & \multicolumn{2}{|c|}{\text{$\gamma=\frac{2}{\alpha+1}+1$}}  \\ \hline
N    & \mathcal{E}_{CN} & rate^{k}_{CN} & \mathcal{E}_{CN} & rate^{k}_{CN} & \mathcal{E}_{CN} & rate^{k}_{CN}  \\ \hline
16   & $1.2098e-02$  & *      & $1.1501e-03$  & *      & $8.3435e-04$  & *        \\
32   & $4.9207e-03$  & 1.30   & $2.9352e-04$  & 1.97   & $2.2642e-04$  & 1.88    \\
64   & $2.0432e-03$  & 1.27   & $7.9232e-05$  & 1.89   & $5.6305e-05$  & 2.01    \\
128  & $8.6517e-04$  & 1.24   & $2.1719e-05$  & 1.87   & $1.0427e-05$  & 2.43    \\
256  & $3.7105e-04$  & 1.22   & $5.9279e-06$  & 1.87   & $2.0922e-06$  & 2.32    \\
\hline
\end{array}
\]
\end{center}
\end{table}

\begin{table}
\caption{Example \ref{Ex1}: $L_2$-norm errors and time convergence rates when $\kappa=2$, $\gamma=\frac{2}{\alpha+1}$ and $M=512$.}\label{eqtable2}
\begin{center} \footnotesize
\renewcommand\arraystretch{1.5}
\noindent\[
\begin{array}{|c||c|c|c|c|c|c|}
\hline
     & \multicolumn{2}{|c|}{\text{$\alpha_1=0.15$, $\alpha_2=0.85$}} & \multicolumn{2}{|c|}{\text{$\alpha_1=0.10$, $\alpha_2=0.20$}} & \multicolumn{2}{|c|}{\text{$\alpha_1=0.80$, $\alpha_2=0.90$}}  \\ \hline
N    & \mathcal{E}_{CN} & rate^{k}_{CN} & \mathcal{E}_{CN} & rate^{k}_{CN} & \mathcal{E}_{CN} & rate^{k}_{CN}  \\ \hline
8    & $4.8275e-03$  & *      & $8.0324e-03$  & *      & $4.7639e-03$  & *        \\
16   & $1.1850e-03$  & 2.03   & $1.9166e-03$  & 2.07   & $8.8628e-04$  & 2.43    \\
32   & $3.0304e-04$  & 1.97   & $4.7020e-04$  & 2.03   & $1.6001e-04$  & 2.47    \\
64   & $8.1431e-05$  & 1.90   & $1.2227e-04$  & 1.94   & $3.3122e-05$  & 2.27    \\
128  & $2.2048e-05$  & 1.88   & $3.2098e-05$  & 1.93   & $8.0217e-06$  & 2.05    \\
\hline
\end{array}
\]
\end{center}
\end{table}

\begin{table}
\caption{Example \ref{Ex1}:  $L_2$-norm errors and time convergence rates when $\alpha_1=\alpha_2=0.5$, $\gamma=\frac{2}{\alpha+1}$ and $M=512$.}\label{eqtable3}
\begin{center} \footnotesize
\renewcommand\arraystretch{1.5}
\noindent\[
\begin{array}{|c||c|c|c|c|c|c|}
\hline
     & \multicolumn{2}{|c|}{\text{$\kappa=0.2$}} & \multicolumn{2}{|c|}{\text{$\kappa=1$}} & \multicolumn{2}{|c|}{\text{$\kappa=5$}}  \\ \hline
N    & \mathcal{E}_{CN} & rate^{k}_{CN} & \mathcal{E}_{CN} & rate^{k}_{CN} & \mathcal{E}_{CN} & rate^{k}_{CN}  \\ \hline
8    & $8.1009e-03$  & *      & $7.6310e-03$  & *      & $5.5947e-03$  & *        \\
16   & $1.7039e-03$  & 2.25   & $1.6545e-03$  & 2.20   & $1.4190e-03$  & 1.98    \\
32   & $3.9316e-04$  & 2.12   & $3.8782e-04$  & 2.09   & $3.6162e-04$  & 1.97    \\
64   & $1.0427e-04$  & 1.91   & $1.0361e-04$  & 1.90   & $1.0037e-04$  & 1.85    \\
128  & $2.8791e-05$  & 1.86   & $2.8701e-05$  & 1.85   & $2.8255e-05$  & 1.83    \\
256  & $8.0558e-06$  & 1.84   & $8.0397e-06$  & 1.84   & $7.9604e-06$  & 1.83    \\
\hline
\end{array}
\]
\end{center}
\end{table}

\begin{table}
\caption{Example \ref{Ex1}: Numerical solutions when $\kappa=1$, $\gamma=\frac{2}{\alpha+1}$ and $M=256$.}\label{eqtable4}
\begin{center} \footnotesize
\renewcommand\arraystretch{1.5}
\noindent\[
\begin{array}{|c||c|c|c|c|c|c|}
\hline
     & \multicolumn{2}{|c|}{\text{$\alpha_1=0.15$, $\alpha_2=0.85$}} & \multicolumn{2}{|c|}{\text{$\alpha_1=0.10$, $\alpha_2=0.20$}} & \multicolumn{2}{|c|}{\text{$\alpha_1=0.80$, $\alpha_2=0.90$}}  \\ \hline
N    & U^{*} & U^{**} & U^{*} & U^{**} & U^{*} & U^{**}  \\ \hline
8    & $3.2498e-02$  & $1.9364e-02$   & $3.2506e-02$  & $2.2373e-02$  & $3.2499e-02$  & $1.6138e-02$      \\
16   & $3.2512e-02$  & $1.8805e-02$   & $3.2514e-02$  & $2.1881e-02$  & $3.2512e-02$  & $1.5631e-02$   \\
32   & $3.2515e-02$  & $1.8518e-02$   & $3.2516e-02$  & $2.1626e-02$  & $3.2516e-02$  & $1.5373e-02$   \\
64   & $3.2516e-02$  & $1.8374e-02$   & $3.2516e-02$  & $2.1497e-02$  & $3.2516e-02$  & $1.5243e-02$   \\
128  & $3.2517e-02$  & $1.8301e-02$   & $3.2517e-02$  & $2.1432e-02$  & $3.2517e-02$  & $1.5177e-02$   \\
256  & $3.2517e-02$  & $1.8264e-02$   & $3.2517e-02$  & $2.1399e-02$  & $3.2517e-02$  & $1.5145e-02$   \\
512  & $3.2517e-02$  & $1.8246e-02$   & $3.2517e-02$  & $2.1382e-02$  & $3.2517e-02$  & $1.5129e-02$    \\
\hline
\end{array}
\]
\end{center}
\end{table}

\begin{table}
\caption{Example \ref{Ex1}: Numerical solutions when $\alpha_1=\alpha_2=0.5$, $\gamma=\frac{2}{\alpha+1}$ and $M=32$.}\label{eqtable5}
\begin{center} \footnotesize
\renewcommand\arraystretch{1.5}
\noindent\[
\begin{array}{|c||c|c|c|c|c|c|}
\hline
     & \multicolumn{2}{|c|}{\text{$\kappa=0$}} & \multicolumn{2}{|c|}{\text{$\kappa=0.5$}} & \multicolumn{2}{|c|}{\text{$\kappa=1$}}  \\ \hline
N    & U^{*} & U^{**} & U^{*} & U^{**} & U^{*} & U^{**}  \\ \hline
32   & $2.5016e-01$  & $1.0454e-01$   & $1.5173e-01$  & $7.3025e-02$  & $9.2033e-02$  & $5.1659e-02$   \\
64   & $2.5018e-01$  & $1.0231e-01$   & $1.5175e-01$  & $7.1730e-02$  & $9.2040e-02$  & $5.0921e-02$   \\
128  & $2.5019e-01$  & $1.0119e-01$   & $1.5175e-01$  & $7.1080e-02$  & $9.2042e-02$  & $5.0548e-02$   \\
256  & $2.5019e-01$  & $1.0063e-01$   & $1.5175e-01$  & $7.0753e-02$  & $9.2043e-02$  & $5.0360e-02$   \\
512  & $2.5019e-01$  & $1.0035e-01$   & $1.5175e-01$  & $7.0590e-02$  & $9.2043e-02$  & $5.0266e-02$   \\
1024 & $2.5019e-01$  & $1.0021e-01$   & $1.5175e-01$  & $7.0509e-02$  & $9.2043e-02$  & $5.0219e-02$    \\
2048 & $2.5019e-01$  & $1.0014e-01$   & $1.5175e-01$  & $7.0468e-02$  & $9.2043e-02$  & $5.0195e-02$    \\
\hline
\end{array}
\]
\end{center}
\end{table}

\vskip 0.05in
\begin{example} \label{Ex2}  In this example, we consider the two-dimensional (2D) case ($d=2$) of \eqref{eq1.1} with $m=2$. Let $A$ be the 2D Laplace operator, and the operators $B_1=-\frac{\partial^2}{\partial x_1^2}$ and $B_2=-\frac{\partial^2}{\partial x_2^2}$ over the domain $\Omega=(0,L)\times(0,L)$ with the homogeneous DBCs. In order to satisfy the regularity assumption \eqref{eq2.3}, we present the exact solution of \eqref{eq1.1} as follows
    \begin{equation*}
              u(x_1,x_2,t) =  \left( t^{1+\min\{\alpha_1,\alpha_2\}} + 1 \right) e^{-\kappa t}\sin \pi x_1 \sin \pi x_2,
   \end{equation*}
then the initial condition
    \begin{equation*}
              u_0(x_1,x_2)=\sin \pi x_1 \sin \pi x_2,
   \end{equation*}
and the source term $f(x_1,x_2,t)$ can be calculated correspondingly.
\end{example}

Here, in Table \ref{eqtable6} we present the $L_2$-norm errors and time convergence rates when fixing $\kappa=2$, $\gamma=\frac{2}{\alpha+1}$ and $M=70$. The numerical results incarnate time second-order convergence by selecting three kinds of values of $\alpha_1$ and $\alpha_2$, when $N$ increases gradually.

Fixing $\alpha_1=0.3$, $\alpha_2=0.6$, $\gamma=\frac{2}{\alpha+1}+\frac{1}{2}$ and $M=80$, the numerical results from Table \ref{eqtable7} approximate second order in the time direction, with different tempered parameter $\kappa$ $(\kappa=0,0.2,1,5)$, which is in accordance with the theory.

\begin{table}
\caption{Example \ref{Ex2}:  $L_2$-norm errors and time convergence rates when $\kappa=2$, $\gamma=\frac{2}{\alpha+1}$ and $M=70$.}\label{eqtable6}
\begin{center} \footnotesize
\renewcommand\arraystretch{1.5}
\noindent\[
\begin{array}{|c||c|c|c|c|c|c|}
\hline
     & \multicolumn{2}{|c|}{\text{$\alpha_1=0.10$, $\alpha_2=0.90$}} & \multicolumn{2}{|c|}{\text{$\alpha_1=0.15$, $\alpha_2=0.20$}} & \multicolumn{2}{|c|}{\text{$\alpha_1=0.80$, $\alpha_2=0.75$}}  \\ \hline
N    & \mathcal{E}_{CN} & rate^{k}_{CN} & \mathcal{E}_{CN} & rate^{k}_{CN} & \mathcal{E}_{CN} & rate^{k}_{CN}  \\ \hline
12    & $1.3009e-02$  & *      & $1.8510e-02$  & *      & $2.8330e-03$  & *        \\
24    & $3.8048e-03$  & 1.77   & $5.5602e-03$  & 1.74   & $8.7227e-04$  & 1.70    \\
48    & $1.0148e-03$  & 1.91   & $1.5145e-03$  & 1.88   & $2.4016e-04$  & 1.86    \\
96    & $2.6995e-04$  & 1.91   & $4.0689e-04$  & 1.90   & $6.2739e-05$  & 1.94    \\
\hline
\end{array}
\]
\end{center}
\end{table}

\begin{table}
\caption{Example \ref{Ex2}:  $L_2$-norm errors and time convergence rates when $\alpha_1=0.3$, $\alpha_2=0.6$, $\gamma=\frac{2}{\alpha+1}+\frac{1}{2}$ and $M=80$.}\label{eqtable7}
\begin{center} \footnotesize
\renewcommand\arraystretch{1.5}
\noindent\[
\begin{array}{|c||c|c|c|c|c|c|c|c|}
\hline
     & \multicolumn{2}{|c|}{\text{$\kappa=0$}} & \multicolumn{2}{|c|}{\text{$\kappa=0.2$}} & \multicolumn{2}{|c|}{\text{$\kappa=1$}} & \multicolumn{2}{|c|}{\text{$\kappa=5$}}  \\ \hline
N    & \mathcal{E}_{CN} & rate^{k}_{CN}  & \mathcal{E}_{CN} & rate^{k}_{CN} & \mathcal{E}_{CN} & rate^{k}_{CN} & \mathcal{E}_{CN} & rate^{k}_{CN}  \\ \hline
4    & $2.3908e-02$  & *     & $2.3711e-02$  & *      & $2.2942e-02$  & *      & $1.9567e-02$  & *       \\
8    & $5.8707e-03$  & 2.03  & $5.8643e-03$  & 2.01   & $5.8393e-03$  & 1.97   & $5.7221e-03$  & 1.77    \\
16   & $1.3915e-03$  & 2.08  & $1.3904e-03$  & 2.08   & $1.3858e-03$  & 2.07   & $1.3634e-03$  & 2.07   \\
32   & $3.2342e-04$  & 2.10  & $3.2318e-04$  & 2.10   & $3.2219e-04$  & 2.10   & $3.1738e-04$  & 2.10   \\
64   & $1.0394e-04$  & 1.64  & $8.5550e-05$  & 1.92   & $6.6420e-05$  & 2.28   & $6.5801e-05$  & 2.27    \\
\hline
\end{array}
\]
\end{center}
\end{table}

\section{Concluding remarks} \label{sec6}
In this work, we have considered and analyzed numerical solutions for Volterra integrodifferential equations with tempered multi-term kernels. First we deduced certain regularity estimates of the exact solution of \eqref{eq1.1}. Then under graded meshes, we applied the Crank-Nicolson method and PI rule to construct a time discrete scheme. Based on the regularity assumptions, we proved the unconditional stability and accurate second-order convergence for time by the energy argument. Finally, numerical experiments verified our theoretical results.

\vskip 0.2mm
Moreover, the regularity of exact solution of \eqref{eq1.1} can be extended to a semilinear case (source term $f(t)$ replaced by $f(t,u(t))$) by certain appropriate assumptions; also, theoretical results about new numerical scheme could be similarly yielded by adding a Lipschitz condition $|f(t,u_1)-f(t,u_2)|\leq \mathcal{L}|u_1-u_2|$.

\vskip 0.2mm
Note that in this paper we only consider the discrete scheme and numerical analysis in the time direction. In our future work, the fully discrete scheme can be considered combined with some high-precision spatial-discrete techniques, such as local discontinuous Galerkin method, finite difference method, orthogonal spline collocation method, compatible wavelet, etc.

\section*{Conflict of Interest Statement}
The authors declare that they do not have any conflicts of interest.

\section*{Acknowledgements}
The first author would like to thank the reviewers for their helpful suggestions and comments to improve the quality of this paper. In addition, the first author is also very grateful to his girlfriend Dr. Kexin Li for her support in scientific research and care in life.


\end{document}